\documentclass{amsart}
\usepackage{framed, color}
\definecolor{shadecolor}{rgb}{1,0.8,0.3}

\newtheorem{theorem}{Theorem}[section]
\newtheorem{lemma}[theorem]{Lemma}

\theoremstyle{definition}

\theoremstyle{remark}
\newtheorem{remark}[theorem]{Remark}

\numberwithin{equation}{section}

\begin{document}

\title[Kazdan-Warner equations]{Existence of solutions to a class of Kazdan-Warner equations on finite graphs}


\author{Yi Li}
\address{School of Mathematics and Shing-Tung Yau Center of Southeast University, Southeast University, Nanjing, 21189, China}
\email{yilicms@gmail.com, yilicms@seu.edu.cn}

\author{Qianwei Zhang}
\address{School of Mathematics, Southeast University, Nanjing, 21189, China}
\email{qianweizhang@seu.edu.cn}


\keywords{}

\date{}

\dedicatory{}

\begin{abstract}
Let $G=(V, E)$ be a connected finite graph, $h$ be a positive function on $V$ and $\lambda _{1}(V)$ be the first non-zero eigenvalue of $-\Delta$. For any given finite measure $\mu$ on $V$, define functionals 
\begin{eqnarray*}
    J_{ \beta }(u)&=&\frac{1}{2}\int_{V}|\nabla u|^{2}d \mu -\beta \log\int_{V}he^{u}d \mu,\\ 
    J_{ \alpha ,\beta }(u)&=&\frac{1}{2}\int_{V}\left(|\nabla u|^{2}- \alpha u^{2}\right) d \mu -\beta \log\int_{V}he^{u}d \mu 
\end{eqnarray*}
on the functional space
$$
{\bf H}= \left\{ u\in{\bf W}^{1,2}(V) \Bigg| \int_{V}u\!\ d\mu =0 \right\}. 
$$

For any $\beta \in  \mathbb{R}$,  we show that $J_{ \beta }(u)$ has a minimizer $u\in{\bf H}$, and then, based on variational principle, the Kazdan-Warner equation 
$$
\Delta u=-\frac{\beta he^{u}}{\displaystyle{\int_{V}he^{u}d \mu }}+\frac{\beta }{\text{Vol}(V)}
$$
has a  solution in ${\bf H}$. 

If $\alpha <  \lambda _{1}(V)$, then for any $\beta \in  \mathbb{R} , J_{ \alpha ,\beta }(u)$ has a minimizer in ${\bf H}$, thus the Kazdan-Warner equation
$$
\Delta u+\alpha\!\ u=-\frac{\beta he^{u}}{\displaystyle{\int_{V}he^{u}d \mu }}+\frac{\beta }{\text{Vol}(V)}
$$
has a  solution in ${\bf H}$.  If $\alpha > \lambda _{1}(V)$, then for any $\beta \in  \mathbb{R}$, $\displaystyle{\inf_{u\in{\bf H}} J_{ \alpha ,\beta }(u) =- \infty}$. When $\alpha=\lambda_{1}(V)$, the situation becomes complicated: if $\beta=0$, the corresponding equation is $-\Delta u=\lambda_{1}(V)u$ which has a solution in ${\bf H}$ obviously; if $\beta>0$, then $\displaystyle{\inf_{u\in {\bf H}} J_{\alpha,\beta }(u) =- \infty}$; if $\beta<0$, $J_{ \alpha ,\beta }(u)$ has a minimizer in some subspace of ${\bf H}$. 

Moreover, we consider the same problem where higher eigenvalues are involved. 
\end{abstract}

\maketitle

\section{Introduction}

Suppose that $(\Sigma ,g)$ is a compact Riemannian surface without boundary, $h$ is a positive smooth function on $\Sigma$ and ${\bf W}^{1,2}(\Sigma)$ is the usual Sobolev space. Define the following functional space, where $dv_{g}$ is the area form of $g$,
\begin{equation*}
    H=\left \{u\in{\bf W}^{1,2}(\Sigma) \Bigg| \int_{\Sigma}u\!\ dv_{g}=0\right \}.
\end{equation*}
Define functionals 
\begin{eqnarray*}
    J_{ \beta }(u)&=&\frac{1}{2}\int_{\Sigma}|\nabla_{g} u|^{2}dv_{g} -\beta \log\int_{\Sigma}he^{u}dv_{g},\\
    J_{ \alpha ,\beta }(u)&=&\frac{1}{2}\int_{\Sigma}\left(|\nabla u|^{2}- \alpha u^{2}\right)dv_{g} -\beta \log\int_{\Sigma}he^{u}dv_{g}
\end{eqnarray*}
where $\nabla_{g} u$ denotes the gradient of $u$. 
In view of \cite{curvature}, if the functional $J_{ \beta }(u)$ or $ J_{ \alpha ,\beta }(u)$ has a minimizer $u\in{\bf H}$, then based on the method of Lagrange multipliers the minimizer is a solution of the Kazdan-Warner equation
\begin{equation}
    \Delta_{g}u=-\frac{\beta he^{u}}{\displaystyle{\int_{\Sigma}he^{u}dv_{g} }}+\frac{\beta }{\text{Vol}(\Sigma)},
\label{kw11}
\end{equation}
or
\begin{equation}
     \Delta_{g} u+\alpha u=-\frac{\beta he^{u}}{\displaystyle{\int_{\Sigma}he^{u}d \mu }}+\frac{\beta }{\text{Vol}(\Sigma)},
\label{kw22} 
\end{equation}
where $\Delta_{g}$ is the Laplace-Beltrami operator of $g$. Considering the Trudinger-Moser inequality in \cite{borderline}, we see that $J_{ \beta }(u)$ has a minimizer in ${\bf H}$ for any $\beta < 8\pi $. When $\beta = 8\pi$, Ding et al. in \cite{equation} proved that under some geometric condition $J_{8\pi }(u)$ has a minimizer in ${\bf H}$. However, if $\beta > 8\pi $, then we have $\displaystyle{\inf_{u\in{\bf H}} J_{\beta }(u) =- \infty}$.
Let $\lambda _{1}(\Sigma)$ be the first non-zero eigenvalue of $ -\Delta _{g}$. Referring to \cite{surface}, we have the following assertions:
(i) If $\beta < 8\pi $, then for any $\alpha \in \mathbb{R}$, $J_{\alpha ,\beta }(u)$ has a minimizer in ${\bf H}$;
(ii) If $\beta > 8\pi $, then for any $\alpha \in \mathbb{R}$, we have $\displaystyle{\inf_{u\in {\bf H}} J_{\alpha ,\beta }(u) =- \infty}$;
(iii) If $\alpha <  \lambda _{1}(\Sigma )$, then under some geometric condition $ J_{ \alpha ,8\pi }(u)$ has a minimizer in ${\bf H}$; If $\alpha \geq \lambda _{1}(\Sigma )$, then $\displaystyle{\inf_{u\in{\bf H}} J_{\alpha ,8\pi  }(u) =- \infty}$.
As in \cite{extremal, blow}, $\alpha $ is allowed to be larger than higher eigenvalues. And the same problem has similar results (see \cite{surface}). 

${}$

Several people have studied Kazdan-Warner equation  
\begin{equation}
    \Delta u=c-he^{u}
\label{c}
\end{equation}
on graphs. In \cite{graph}, Grigor'yan, Lin and Yang have given some necessary and sufficient conditions for the existence of solutions to \eqref{c} on finite graphs when $c >0$ or $c=0$. However, when $c<0$, the situation is much more complicated: If the solution exists, then $\overline{h}<0$, where $\overline{h}$ is the average of $h$; conversely, if $\overline{h}<0$, then one can find solutions when $c$ lies in $(c_{-}(h), 0)$ for some negative number $c_{-}(h)$ (see also \cite{graph}). In Ge's paper \cite{negative},  the case of $c=c_{-}(h)$ is under consideration and the corresponding equation of \eqref{c} has a solution. In \cite{infinite}, the Kazdan-Warner equation \eqref{c} is considered on infinite graphs and one can find some sufficient conditions after using a heat method which refered to \cite{heat}.
Furthermore, Ge also extended to the $p$-Laplacian on graph and studied the $p$-th Kazdan-Warner equation $\Delta_{p}(u)=c-he^{u}$ on a finite graph \cite{p}. 

${}$

In this paper, we consider existence of solutions to the Kazdan-Warner equations $\eqref{kw11}$ and $\eqref{kw22}$ on finite graphs. One can find that on compact Riemann surface, $\beta$ are assumed to be greater than 0. However, we intend to consider the case of $\beta \leq 0$ on finite graph. In our setting, we shall prove the following : 
(i) For any $\beta \in  \mathbb{R}$,  $ J_{ \beta }(u)$ has a minimizer $u\in{\bf H}$;
(ii) If $\alpha <  \lambda _{1}(V)$, then for any $\beta \in  \mathbb{R} , J_{ \alpha ,\beta }(u)$ has a minimizer $u$ in ${\bf H}$; If $\beta=0$, the corresponding equation is $-\Delta u=\lambda_{1}(V)u$ which has a solution in $H$ obviously. If $\beta>0$, then $\displaystyle{\inf_{u\in {\bf H}} J_{ \lambda_{1}(V) ,\beta }(u) =- \infty}$. If $\beta<0$, $J_{ \alpha ,\beta }(u)$ has a minimizer in some subspace of $H$. 
(iii) The similar results establish where higher eigenvalues are involved.

We organized this paper as follows: In Section 2, we introduce some notations on graphs and state our main results. In Section 3--7, we prove our results.

Introduction

\section{Notations and results}

Let $G=(V, E)$ be a finite graph, where $V$ denotes the vertex set and $E$ denotes the edge set.
Throughout this paper, all graphs are assumed to be connected. For any edge $xy\in E$, we assume that its weight $w_{xy}> 0$ and that $w_{xy}=w_{yx}$. Let $\mu :V\rightarrow \mathbb{R}^{+}$ be a finite measure. For any function $u :V\rightarrow \mathbb{R}$, the $\mu$-Laplacian (or Laplacian for short) of $u$ is defined by
\begin{equation*}
    \Delta u(x)=\frac{1}{\mu (x)}\sum_{y\sim x}w_{xy}\left[u(y)-u(x)\right],
\end{equation*}
where $y\sim x$ means $xy\in E$. The associated gradient form reads
\begin{equation*} 
    \Gamma (u,v)=\frac{1}{2\mu (x)}\sum_{y\sim x}w_{xy}\left[u(y)-u(x)\right]\left[v(y)-v(x)\right].
\end{equation*}
Write $ \Gamma (u)= \Gamma (u,u)$. We denote the length of its gradient by
\begin{equation*}
   |\nabla u|(x)=\sqrt{\Gamma (u(x))}=\left\{\frac{1}{2\mu (x)}\sum_{y\sim x}w_{xy}\left[u(y)-u(x)\right]^{2}\right\}^{\frac{1}{2}}.
\end{equation*}
For any function $g :V\rightarrow \mathbb{R}$, an integral of $g$ over $V$ is defined by 
\begin{equation}
    \int_{V}g\!\ d\mu =\sum_{x\in V}\mu (x)g(x).
\label{int}
\end{equation}
We let
$$
\text{Vol}(V)=\sum_{x\in V}\mu (x)
$$
be the volume of $V$ relative to the measure $\mu$.
Let ${\bf F}(V)$ be the space of all functions on $V$.
Define the Sobolev space ${\bf W}^{1,2}(V)$ and the associated norm by 
\begin{equation*}
    {\bf W}^{1,2}(V)=\left \{ u \in {\bf F}(V)\Bigg|\int_{V}\left(|\nabla u|^{2}+u^{2}\right)d \mu < + 
 \infty  \right \} 
\end{equation*}
and 
\begin{equation*}
    \left \|u \right \|_{{\bf W}^{1,2}(V)}=\left[\int_{V}\left(|\nabla u|^{2}+u^{2}\right)d \mu \right]^{\frac{1}{2}},
\end{equation*}
respectively. Define a subspace
\begin{equation}
    {\bf H}= \left\{ u\in{\bf W}^{1,2}(V)\Bigg| \int_{V}u\!\ d\mu =0 \right \}
\label{H}
\end{equation}
of ${\bf W}^{1,2}(V)$. 

Let $h$ be a given positive function on $V$ and $\beta\in\mathbb{R}$ be a given constant.
We first consider the functional $J_{ \beta }: {\bf W}^{1,2}(V) \rightarrow  \mathbb{R}$ defined by
\begin{equation}
    J_{ \beta }(u)=\frac{1}{2}\int_{V}|\nabla u|^{2}d \mu -\beta \log\int_{V}he^{u}d \mu.
\label{J_beta}    
\end{equation}

Our first result is the following.
\begin{theorem}
Let $G=(V, E)$ be a finite graph, h be a positive function on V, and H and $J_{ \beta }$ be defined as in \eqref{H} and \eqref{J_beta} ,respectively. Then for any $\beta \in  \mathbb{R}$,  $J_{ \beta }$ has a minimizer $u\in{\bf H}$. Moreover, the Kazdan-Warner equation
\begin{equation}
     \Delta u=-\frac{\beta he^{u}}{\displaystyle{\int_{V}he^{u}d \mu}}+\frac{\beta }{{\rm Vol}(V)}
\label{kw_1} 
\end{equation}
has a  solution $u\in{\bf H}$.
\label{Th1}
\end{theorem}

Let $\lambda _{1}(V)$ be the first non-zero eigenvalue of $-\Delta$, say 
\begin{equation}
    \lambda _{1}(V)= \inf_{u\in{\bf H}, \int_{V}u^{2}d\mu=1} \int_{V}|\nabla u|^{2}d \mu.
\label{lamda1}
\end{equation}
If $\alpha <  \lambda _{1}(V)$, then
\begin{equation*}
      \left \|u \right \|_{1, \alpha }:=\left[\int_{V}\left(|\nabla u|^{2}- \alpha u^{2}\right)d \mu\right]^{\frac{1}{2}}
\end{equation*}
defines a Sobolev norm on ${\bf H}$. Since ${\bf H}$ is a finite dimensional linear space, $  \left \| \cdot \right \|_{1, \alpha }$ is equivalent to the Sobolev norm $\left \|  \cdot \right \|_{{\bf W}^{1,2}(V)}$. 

To achieve an analog of \eqref{kw_1}, we turn to the following functional  
\begin{equation}
     J_{ \alpha ,\beta }(u)=\frac{1}{2}\int_{V}\left(|\nabla u|^{2}- \alpha u^{2}\right) d \mu -\beta \log\int_{V}he^{u}d\mu 
\label{J_a,b} 
\end{equation}
where $\alpha , \beta$ are given constants.
Now we state an analog of Theorem \ref{Th1} as follows.

\begin{theorem}{\bf ($\boldsymbol{J_{\alpha, \beta}$: $\alpha<\lambda_{1}(V)}$)}
Let $G=(V, E)$ be a finite graph, h be a positive function on V, and ${\bf H}$, $\lambda _{1}(V)$ and $ J_{ \alpha ,\beta }$ be defined as in \eqref{H}, \eqref{lamda1} and \eqref{J_a,b}, respectively. Then for any $\alpha <  \lambda _{1}(V), \beta \in  \mathbb{R}$, $J_{ \alpha ,\beta }$ has a minimizer $u$ in ${\bf H}$. Moreover, the Kazdan-Warner equation
\begin{equation}
     \Delta u+\alpha u=-\frac{\beta he^{u}}{\displaystyle{\int_{V}he^{u}d \mu}}+\frac{\beta }{{\rm Vol}(V)}
\label{kw_2} 
\end{equation}
has a  solution $u\in{\bf H}$.
\label{Th2}
\end{theorem}

Let $\lambda _{1}(V)< \lambda _{2}(V)< \cdots< \lambda_{m-1}(V)$ be all distinct non-zero eigenvalues of $-\Delta$,  $E_{\lambda _{k}(V)}$ be the eigenfuction space with respect to $\lambda _{k}(V)$, $\{u_{ki}\}_{i=1}^{n_{k}}$ be an orthonormal basis of $ E_{\lambda _{k}(V)} (k=1, \cdots, m-1)$. Define 
\begin{equation}
\begin{aligned}
     E_{k}&:=E_{\lambda _{1}(V)} \bigoplus \cdots \bigoplus E_{\lambda _{k}(V)},\\
    E_{k}^{\perp } &:= \text{the complement of} \ E_{k} \ = \ \left \{ u \in{\bf H} \Bigg|\int_{V}uv\!\ d \mu =0, \ \forall \ v \in E_{k}\right \}.
\label{E_kp} 
\end{aligned}  
\end{equation}
Now we consider the case of $\alpha \geq \lambda _{1}(V)$. 
If $\alpha > \lambda _{1}(V)$, one can find $\inf_{u\in {\bf H}} J_{ \alpha ,\beta }(u) =- \infty $. However, when $\alpha = \lambda _{1}(V)$, our conclusion becomes complicated.

\begin{theorem}{\bf ($\boldsymbol{J_{\alpha, \beta}$: $\alpha\geq \lambda_{1}(V)}$)}\label{Th3}
 Let $G=(V, E)$, $h$, ${\bf H}$, $\lambda _{1}(V)$ and $ J_{ \alpha ,\beta }$ be as in Theorem \ref{Th2}, $E_{1}^{\perp}$ be defined as in \eqref{E_kp}. Then we have the following assertions:
 
\begin{itemize}

\item[(i)] When $\alpha > \lambda _{1}(V)$, for any $\beta \in  \mathbb{R}$, we have
$$
\inf_{u\in {\bf H}} J_{ \alpha ,\beta }(u) =- \infty.
$$

\item[(ii)] When $\alpha=\lambda_{1}(V)$, we have the following three subcases:

\begin{itemize} 

\item[(ii-a)] For any $ \beta>0 $, we have
$$
\inf_{u\in {\bf H}} J_{\lambda _{1}(V)  ,\beta }(u) =- \infty;
$$

\item[(ii-b)] $J_{ \lambda _{1}(V) ,0 }$ has a minimizer $u$ in ${\bf H}$ so that the Kazdan-Warner equation \eqref{kw_2} has a solution $u\in{\bf H}$; 

\item[(ii-c)] For any $\beta <0$, $J_{ \lambda _{1}(V) ,\beta}$ has a minimizer $u$ in $E_{1}^{\perp}$ so that the Kazdan-Warner equation 
\begin{equation}
   \Delta u+\alpha u=-\frac{\beta he^{u}}{\displaystyle{\int_{V}he^{u}d \mu }}+\frac{\beta }{{\rm Vol}(V)}+\sum_{i=1}^{n_{1}}\frac{\displaystyle{\beta \int_{V}hu_{1i}e^{u}d\mu}}{\displaystyle{\int_{V}he^{u}d\mu}} u_{1i}
\label{kw_4} 
\end{equation}
\end{itemize}
\end{itemize}
has a solution $u\in E_{1}^{\perp}$.

\end{theorem}

\begin{remark}
   In fact, the corresponding Kazdan-Warner equation of $J_{\lambda _{1}(V), 0}$ is $-\Delta u=\lambda_{1}(V)u$ which has a solution obviously. 
\end{remark}

Let $\lambda_{k+1}(V)$ be the $(k+1)$-th non-zero eigenvalue of $-\Delta$, say 
\begin{equation}
    \lambda _{k+1}(V)= \inf_{u\in E_{k}^{\perp }, \int_{V}u^{2}d\mu=1} \int_{V}|\nabla u|^{2}d \mu .
\label{lambdak+1}
\end{equation}
We consider higher eigenvalue case below.

\begin{theorem}{\bf ($\boldsymbol{J_{\alpha, \beta}$: $\alpha< \lambda_{k+1}(V)}$)}
Let $G=(V, E)$, $h$, ${\bf H}$ and $ J_{ \alpha ,\beta }$ be as in Theorem \ref{Th3}, $\lambda _{k}(V)$ be the k-th non-zero eigenvalue of $-\Delta$, $E_{\lambda _{k}(V)}$ be the eigenfuction space with respect to $\lambda _{k}(V)$, $\{u_{ki}\}_{i=1}^{n_{k}}$ be an orthonormal basis of $ E_{\lambda _{k}(V)}$, $E_{k}$ and $ E_{k}^{\perp }$ be defined as in \eqref{E_kp}. Then for any $\alpha <  \lambda _{k+1}(V), \beta \in  \mathbb{R} , J_{ \alpha ,\beta }$ has a minimizer $u$ in $E_{k}^{\perp }$.  Moreover, the Kazdan-Warner equation 
\begin{equation}
     \Delta u+\alpha u=-\frac{\beta he^{u}}{\displaystyle{\int_{V}he^{u}d \mu }}+\frac{\beta }{{\rm Vol}(V)}+\sum_{s=1}^{k}\sum_{i=1}^{n_{s}}\frac{\displaystyle{\beta \int_{V}hu_{si}e^{u}d\mu}}{\displaystyle{\int_{V}he^{u}d\mu}} u_{si}
\label{kw_3}
\end{equation}
has a  solution $u\in  E_{k}^{\perp }$.
\label{Th4}
\end{theorem}

\begin{theorem}{\bf ($\boldsymbol{J_{\alpha, \beta}$: $\alpha\geq \lambda_{k+1}(V)}$)}\label{Th5}
Let $G=(V, E)$, $h$, ${\bf H}$, $J_{ \alpha ,\beta }, \{u_{ki}\}_{i=1}^{n_{k}}$, $\lambda _{k}(V)$, $E_{k},  E_{k}^{\perp }$ be as in Theorem \ref{Th4}. Then we have the following assertions:

\begin{itemize} 

\item[(i)] When $\alpha >\lambda _{k+1}(V)$, for any $\beta \in  \mathbb{R}$, we have
$$
\inf_{u\in E_{k}^{\perp }} J_{ \alpha ,\beta }(u) =- \infty.
$$

\item[(ii)] When $\alpha=\lambda_{k+1}(V)$, we have the following three subcases:

\begin{itemize} 

\item[(ii-a)] For any $\beta>0$, we have
$$
\inf_{u\in E_{k}^{\perp }} J_{ \lambda _{k+1}(V) ,\beta }(u) =- \infty; 
$$

\item[(ii-b)] $J_{ \lambda _{k+1}(V) ,0 }$ has a minimizer $u$ in ${ E_{k}^{\perp }}$ so that the Kazdan-Warner equation \eqref{kw_3} has a solution $u\in E_{k}^{\perp } $;

\item[(ii-c)] For any $\beta <0$, $J_{ \lambda _{k+1}(V) ,\beta}$ has a minimizer $u$ in $E_{k+1}^{\perp}$ so that the Kazdan-Warner equation 
\begin{equation}
   \Delta u+\alpha u=-\frac{\beta he^{u}}{\displaystyle{\int_{V}he^{u}d \mu}}+\frac{\beta }{{\rm Vol}(V)}+\sum_{i=1}^{k+1}\sum_{i=1}^{n_{s}}\frac{\displaystyle{\beta \int_{V}hu_{si}e^{u}d\mu}}{\displaystyle{\int_{V}he^{u}d\mu}} u_{si}
\label{kw_5} 
\end{equation}
has a solution $u\in E_{k+1}^{\perp}$. 

\end{itemize}
\end{itemize}

\end{theorem}

\begin{remark}
   In fact, the corresponding Kazdan-Warner equation of $J_{\lambda _{k+1}(V)  ,0 }$ is $-\Delta u=\lambda_{k+1}(V)u$ which has a solution obviously.
\end{remark}

\section{Proof of Theorem 2.1}

Grigor'yan, Lin and Yang have studied the Kazdan-Warner equation on a finite graph. And they have given some significant lemmas in \cite{graph} which we will review below.

\begin{lemma}{\rm (\cite{graph})}
    Let $G=(V, E)$ be a finite graph. Then the Sobolev space  ${\bf W}^{1,2}(V)$ is pre-compact. Namely, if $\{u_{j}\}_{j}$ is bounded in  ${\bf W}^{1,2}(V)$, then there exists some $ u \in {\bf W}^{1,2}(V)$ such that up to a subsequence, $u_{j}\to u$ in ${\bf W}^{1,2}(V)$.
\label{Le1}
\end{lemma}

\begin{lemma}{\rm (\cite{graph})}
   Let $G=(V, E)$ be a finite graph. For all functions $u\in {\bf H}$, there exists some constant $C_{P}$ depending only on $G$ such that
\begin{equation*}
    \int_{V}u^{2}d\mu \leq C_{{}}\int_{V}|\nabla u|^{2}d\mu.
\end{equation*}
\label{Le2}
\end{lemma}

\begin{lemma}{\rm (\cite{graph})}
     Let $G=(V, E)$ be a finite graph. For any $\theta  \in \mathbb{R} $, there exists a constant $C$ depending only on $\theta $ and $G$ such that for all functions $v\in{\bf H}$ with $\displaystyle{\int_{V}|\nabla v|^{2}d\mu \leq 1}$, there holds
\begin{equation*}
    \int_{V}e^{\theta  v^{2}}d\mu \leq C.
\end{equation*}
\label{Le3}
\end{lemma}

Since the case of $\beta \leq 0$ is considered, we need to estimate the lower bound of $\displaystyle{\int_{V}he^{u}d \mu}$.

\begin{lemma}
    Let $G=(V, E)$ be a finite graph, h be a positive function on $V$. There exist constants $C_{1}$ and $C_{2}$ depending only on $G$ and $h$ such that for any $u\in H$ , there holds
\begin{equation}
    \int_{V}he^{u}d \mu \geq C_{1} \exp{\left(C_{2}\left \| \nabla u\right \|_{L^{2}(V)} \right)} .
\label{intheu}
\end{equation}
\end{lemma}

\begin{proof}
    In view of \eqref{int}, we obtain the following inequality
\begin{equation}
    \left \| u\right \|_{L^{\infty }(V)}\leq \left \|  u\right \|_{L^{2}(V)}\mu _{\min }^{-\frac{1}{2}}.
\label{infty-2}
\end{equation}
By \eqref{infty-2} and Lemma \ref{Le2}, we have
\begin{equation*}
\begin{aligned}
    \int_{V}he^{u}d\mu &\geq  \left ( \min _{x\in V}h(x)\right )\text{Vol}(V) \exp{\left( -\left \| u\right \|_{L^{\infty }(V)} \right)}\\
   &\geq  \left ( \min _{x\in V}h(x)\right )\text{Vol}(V)\exp{\left(-\left \| u\right \|_{L^{2}(V)} \mu _{\min}^{-\frac{1}{2}} \right)} \\
    &\geq  \left ( \min _{x\in V}h(x)\right )\text{Vol}(V)\exp{\left(-\left \| \nabla u\right \|_{L^{2}(V)} C_{P}^{\frac{1}{2}} \mu _{\min}^{-\frac{1}{2}} \right)} . 
\end{aligned}  
\end{equation*}
Hence we obtain  with $\displaystyle{C_{1}=\left(\min_{x\in V}h(x)\right){\rm Vol}(V)}$ and $C_{2}=-C_{P}^{\frac{1}{2}}\mu^{-1/2}_{\min}$.
\end{proof}

${}$

{\it The proof of Theorem 2.1.} We first prove that the functional  $J_{ \beta }(u)$ is bounded in ${\bf H}$.

${}$

{\bf Case $\boldsymbol{1}$: $\boldsymbol{\beta \geq 0}$}. If $\beta=0$, then it is trivial. In the following we may assume $\beta>0 $ and $\text{const} \neq u\in{\bf H}$. Let
$$
\widetilde{u}=\frac{u}{\displaystyle{\left \| \nabla u \right \|_{L^{2}(V)}}}.
$$
Then $\widetilde{u}\in{\bf H}$ and $\left \| \nabla \widetilde u \right\|_{L^{2}(V)}=1$. By Lemma \ref{Le3}, one can find a positive constant $C$ depending only on $\theta $ and $G$ such that 
\begin{equation*}
\int_{V}e^{\theta  \widetilde{u}^{2}}d \mu  \leq C\left ( \theta, G\right ).    
\end{equation*}
For any $\varepsilon > 0$, one has
\begin{equation*}
\begin{aligned}
     \int_{V}e^{u}d\mu &\leq \int_{V} \exp{\left( \varepsilon \left \| \nabla u \right \|_{L^{2}(V)}^{2}+\frac{u^{2}}{4 \varepsilon \left \| \nabla u \right \|_{L^{2}(V)}^{2}}\right)}d\mu\\
      &=\exp{\left( \varepsilon \left \| \nabla u \right \|_{L^{2}(V)}^{2} \right)}\int_{V}\exp{\left( \frac{u^{2}}{4\varepsilon \left \| \nabla u \right \|_{L^{2}(V)}^{2}} \right)}d\mu  \\
    &\leq C\exp{\left( \varepsilon \left \| \nabla u \right \|_{L^{2}(V)}^{2} \right)}  
\end{aligned}
\end{equation*}
where $C$ is a constant depending only on $\varepsilon $ and $G$.
Hence, 
\begin{equation}
    \int_{V}he^{u}d\mu \leq C\left ( \max _{x\in V}h(x)\right )\exp{\left( \varepsilon \left \| \nabla u \right \|_{L^{2}(V)}^{2} \right)}.
\label{he^u11}
\end{equation}
It follows from \eqref{J_beta}, \eqref{he^u11} that 
\begin{equation*}
    J_{ \beta }(u)\geq \frac{1}{2}\int_{V}|\nabla u|^{2}d \mu-\beta \varepsilon \left \| \nabla u \right \|_{L^{2}(V)}^{2}+C_{1}
\end{equation*}
where $C_{1}$ is some constant depending only on $ \varepsilon ,\beta, h$ and $G$.
Choosing $\displaystyle{\varepsilon =\frac{1}{4\beta}}$, we have
\begin{equation}
      J_{ \beta }(u)\geq \frac{1}{4}\int_{V}|\nabla u|^{2}d \mu+C_{1}.
\label{J_b11}
\end{equation}

${}$

{\bf Case $\boldsymbol{2}$: $\boldsymbol{\beta < 0}$}. It follows from \eqref{J_beta} and \eqref{intheu} that 
\begin{equation}
     J_{ \beta }(u)\geq \frac{1}{2}\int_{V}|\nabla u|^{2}d \mu+C_{2}\left \| \nabla u \right \|_{L^{2}(V)}+C_{3}
\label{J_b12}
\end{equation}
where $C_{2}, C_{3}$ are some constants depending only on $ \beta, h$ and $G$. Hence, we obtain
\begin{equation*}
    J_{ \beta }(u)\geq C_{4},
\end{equation*}
where $C_{4}$ is some constant depending only on  $ \beta, h$ and $G$.

${}$

Therefore  $J_{ \beta }(u)$ has a lower bound in ${\bf H}$. This allows us to consider 
\begin{equation*}
    b:=\inf_{u\in{\bf H}} J_{ \beta }(u).
\end{equation*}
Take a sequence of functions $\{u_{j}\}_{j}\subset{\bf H}$ such that $J_{\beta }( u_{j})\rightarrow b$. It follows from \eqref{J_b11} and \eqref{J_b12} that for any $\beta \in \mathbb{R}$, there exists a constant $\widetilde{C}$ depending only on $\beta,h $ and $G$ such that 
\begin{equation*}
    \left \| \nabla u_{j}\right \|_{L^{2}(V)}\leq \widetilde{C}.
\end{equation*}
By Lemma \ref{Le2}, $\{u_{j}\}_{j}$ is bounded in ${\bf W}^{1,2}(V)$. According to Lemma \ref{Le1}, up to a subsequence, $u_{j}\to u$ in ${\bf W}^{1,2}(V)$. It is easy to see that  $u\in{\bf H}$ and $J_{ \beta }(u)= b$. Based on variational principle, we obtain that  for any $\phi \in{\bf H}$, there holds 
\begin{equation*}
  \begin{aligned}
       0&=\frac{d}{dt}\Big{|}_{t=0} J_{\beta }(u+t\phi )\\
      &=\frac{d}{dt}\Big{|}_{t=0}\left ( \frac{1}{2}\int_{V}|\nabla \left ( u+t\phi \right)|^{2}d \mu
       -\beta \log\int_{V}he^{u+t\phi}d \mu\right)  \\
    &=-\int_{V}\left(\Delta u+\frac{\beta he^{u}} {\displaystyle{\int_{V}he^{u}d \mu}}\right)\phi\!\ d\mu.
  \end{aligned}
\end{equation*}
Define by
$$
{\bf H}^{\perp} =\left\{ u \in{\bf F}\Bigg|\int_{V}uv\!\ d \mu =0, \ \forall \ v \in{\bf H}\right\}
$$
the complement space of ${\bf H}$. Then 
\begin{equation}
    \Delta u+\frac{\beta he^{u}}{\displaystyle{\int_{V}he^{u}d \mu}}\in{\bf H}^{\perp }.
\label{belong1}
\end{equation}
We claim that 
\begin{equation}
   H^{\perp }=\left \{\text{const}\right \}.
\label{H^perp}
\end{equation}
Indeed, it suffices to prove that ${\bf H}^{\perp } \subseteq \{\text{const}\}$. Let $v \in{\bf H}^{\perp } $ and $\overline{v}$ be its average over $V$. Since $(v-\overline{v})\in{\bf H}$ and $\overline{v}\in{\bf H}^{\perp}$, it follows that
\begin{equation*}
   0=\int_{V}\left(v-\overline{v}\right)v\!\ d\mu -\int_{V}\left(v-\overline{v}\right)\overline{v}\!\ d\mu=\int_{V}\left(v-\overline{v}\right)^{2}d\mu\geq 0.
\end{equation*}
Hence, we must have $v\equiv \overline{v}$. 
It follows from \eqref{belong1} and \eqref{H^perp} that
\begin{equation}
     \Delta u=-\frac{\beta he^{u}}{\displaystyle{\int_{V}he^{u}d \mu }}+\xi  
\label{xi1}
\end{equation}
where $\xi  $ is some constant. Integrating \eqref{xi1} on both sides, we have
$$
\xi =\frac{\beta}{{\rm Vol}(V)}. 
$$
Therefore, the Kazdan-Warner equation \eqref{kw_1} has a solution $u\in{\bf H}$.


\section{Proof of Theorem 2.2}

To prove Theorem \ref{Th2}, similar to the proof of Theorem \ref{Th1}, we first prove that the functional  $J_{ \alpha ,\beta }(u)$ is bounded in ${\bf H}$.

${}$

{\bf Case $\boldsymbol{1}$: $\boldsymbol{\beta \geq 0}$}. If $\beta=0$, then it is trivial. In the following we may assume $\beta>0$ and $0 \neq u\in{\bf H}$. Let
$$
v=\frac{u}{||u||_{1, \alpha}}. 
$$
Then $v\in{\bf H}$ and $||v||_{1, \alpha }=1$.
We claim that there exists a positive constant $C$ depending only on $\alpha ,\theta $ and $G$ such that 
\begin{equation*}
\int_{V}e^{\theta  v^{2}}d \mu  \leq C\left ( \alpha ,\theta ,G\right ).  
\end{equation*}
Indeed, since $|| \cdot||_{1, \alpha }$ is equivalent to the Sobolev norm $|| \cdot ||_{{\bf W}^{1,2}(V)}$, there exists a constant $K$ depending only on $\alpha $ and $G$ such that for all functions $v\in{\bf H}$,
\begin{equation}
    \int_{V}|\nabla v|^{2}d \mu  \leq K\left \| v \right \|_{1, \alpha }^{2}.
    \label{1,a}
\end{equation}
Let
$$
\widetilde{v}=\frac{v}{\sqrt{K}|| v ||_{1, \alpha }}. 
$$
Then $\widetilde{v} \in  {\bf H}$ and $\displaystyle{\int_{V}|\nabla \widetilde{v}|^{2}d \mu  \leq 1}$. By Lemma \ref{Le3}, one can find a positive constant $C$ depending only on $\theta $ and $G$ such that 
\begin{equation*}
\int_{V}e^{\theta  \widetilde{v}^{2}}d \mu  \leq C\left ( \theta, G\right ).    
\end{equation*}
Hence, we obtain
\begin{equation*} 
   \int_{V}e^{\theta  v^{2}}d \mu =\int_{V}\exp{\left(\theta  \widetilde{v}^{2}\cdot K\left \| v \right \|_{1, \alpha }^{2}\right)}d \mu  \leq e^{K}C\left ( \theta ,G\right ).
\end{equation*}
For any $\varepsilon > 0$, one has
\begin{equation*} 
\begin{aligned}
     \int_{V}e^{u}d\mu &\leq \int_{V}\exp{\left(\varepsilon \left \| u \right \|_{1, \alpha }^{2}+\frac{u^{2}}{4 \varepsilon \left \| u \right \|_{1, \alpha }^{2}}\right)}d\mu\\
     &=\exp{\left(\varepsilon \left \| u \right \|_{1, \alpha }^{2}\right)}\int_{V}\exp{\left(\frac{u^{2}}{4\varepsilon \left \| u \right \|_{1, \alpha }^{2}}\right)}d\mu  \\
     &\leq C\exp{\left(\varepsilon \left \| u \right \|_{1, \alpha }^{2}\right)}
\end{aligned}
\end{equation*}
where $C$ is some constant depending only on $ \alpha ,\varepsilon $ and $G$. Hence
\begin{equation}
    \int_{V}he^{u}d\mu \leq C\left ( \max _{x\in V}h(x)\right )\exp{\left(\varepsilon \left \| u \right \|_{1, \alpha }^{2}\right)}.
\label{he^u21}
\end{equation}
It follows from \eqref{J_a,b} and \eqref{he^u21} that 
\begin{equation*}
    J_{ \alpha ,\beta }(u)\geq \frac{1}{2} \left \| u \right \|_{1, \alpha }^{2}-\beta \varepsilon  \left \| u \right \|_{1, \alpha }^{2}+C_{1}
\end{equation*}
where $C_{1}$ is some constant depending only on $ \varepsilon ,\alpha,\beta, h$ and $G$.
By choosing $\displaystyle{\varepsilon =\frac{1}{4\beta}}$ and considering \eqref{1,a}, we have
\begin{equation}
      J_{ \alpha ,\beta }(u)\geq \frac{1}{4K} \int_{V}|\nabla u|^{2}d \mu+C_{1}.
\label{J_a,b21}
\end{equation}

${}$

{\bf Case $\boldsymbol{2}$: $\boldsymbol{\beta < 0}$}. It follows from \eqref{J_a,b}, \eqref{intheu} and \eqref{1,a} that 
\begin{equation}
     J_{ \alpha ,\beta }(u)\geq \frac{1}{2K}\left \| \nabla u \right \|_{L^{2}(V)}^{2}+C_{2}\left \| \nabla u \right \|_{L^{2}(V)}+C_{3}
\label{J_a,b22}
\end{equation}
where $C_{2},C_{3}$ are some constants depending only on $ \alpha, \beta, h$ and $G$. Hence, we obtain
\begin{equation*}
    J_{ \alpha ,\beta }(u)\geq C_{4}
\end{equation*}
where $C_{4}$ is some constant depending only on $\alpha, \beta, h$ and $G$.

${}$

Therefore  $J_{ \alpha ,\beta }(u)$ has a lower bound on ${\bf H}$. This allows us to consider 
\begin{equation*}
    b=\inf_{u\in H}J_{ \alpha ,\beta }(u).
\end{equation*}
Take a sequence of functions $\{u_{j}\}_{j}\subseteq{\bf H}$ such that $J_{\alpha ,\beta }\left ( u_{j}\right )\rightarrow b$. It follows from \eqref{J_a,b21} and \eqref{J_a,b22} that for any $\beta \in \mathbb{R}$, there exists a constant $\widetilde{C}$ depending only on $\alpha, \beta,h$ and $G$ such that 
\begin{equation*}
    \left \| \nabla u_{j}\right \|_{L^{2}(V)}\leq \widetilde{C}.
\end{equation*}
By Lemma \ref{Le2}, $\{u_{j}\}_{j}$ is bounded in ${\bf W}^{1,2}(V)$. By Lemma \ref{Le1}, up to a subsequence, $u_{j}\to u$ in ${\bf W}^{1,2}(V)$. It is easy to see that  $u\in  {\bf H}$ and $J_{ \alpha ,\beta }(u)= b$.
Based on variational principle, we obtain that for any $\phi \in {\bf H}$, there holds 
\begin{equation*}
  \begin{aligned}
       0&=\frac{d}{dt}\Big{|}_{t=0} J_{\alpha ,\beta }(u+t\phi)\\
       &=\frac{d}{dt}\Big{|}_{t=0}\left ( \frac{1}{2}\int_{V}|\nabla \left ( u+t\phi \right)|^{2}d \mu -\frac{\alpha }{2}\int_{V}(u+t\phi )^{2}d\mu-\beta \log\int_{V}he^{u+t\phi}d \mu\right)\\
    &=-\int_{V}\left(\Delta u+\alpha u+\frac{\beta h e^{u} }{\displaystyle{\int_{V}he^{u}d \mu}}\right)\phi \!\ d\mu .  
  \end{aligned}
\end{equation*}
Hence, we have
\begin{equation*}
    \Delta u+\alpha u+\frac{\beta h e^{u} }{\displaystyle{\int_{V}he^{u}d \mu}} \in{\bf H}^{\perp}.
\end{equation*}
By \eqref{H^perp}, we obtain
\begin{equation}
     \Delta u+\alpha u=-\frac{\beta he^{u}}{\displaystyle{\int_{V}he^{u}d \mu }}+\xi 
\label{xi2}
\end{equation}
where $\xi$ is a constant. Integrating \eqref{xi2} on both sides, we have
$$
\xi =\frac{\beta }{{\rm Vol}(V)}. 
$$
Therefore, the Kazdan-Warner equation \eqref{kw_2} has a solution $u\in{\bf H}$.

\section{Proof of Theorem 2.3}

Recall that $\lambda _{1}(V)< \lambda _{2}(V)< \cdots< \lambda_{m-1}(V)$ are all distinct non-zero eigenvalues of $-\Delta$,  $E_{\lambda _{k}(V)}$ are the eigenfuction space with respect to $\lambda _{k}(V)$, $\left \{u_{ki}\right \}_{i=1}^{n_{k}}$ is an orthonormal basis of $ E_{\lambda _{k}(V)} (k=1, \cdots, m-1)$ and 
\begin{equation*}
\begin{aligned}
     E_{k}&:=E_{\lambda _{1}(V)} \bigoplus \cdots \bigoplus E_{\lambda _{k}(V)},\\
    E_{k}^{\perp } &:= \text{the complement of} \ E_{k} \ = \ \left \{ u \in{\bf H} \Bigg|\int_{V}uv\!\ d \mu =0, \ \forall \ v \in E_{k}\right \}.
\end{aligned}  
\end{equation*}
Now we consider the case of $\alpha \geq \lambda _{1}(V)$.

\subsection{Proof of Theorem 2.3: (i)}

Fix $\alpha > \lambda _{1}(V)$. Take $u_{1}$ as an eigenfunction with respect to $\lambda_{1}(V)$. Then $u_{1}\in {\bf H}$. Obviously, we have
\begin{equation*}
    \int_{V}|\nabla u_{1}|^{2}d \mu=\lambda _{1}(V)\int_{V}u_{1}^{2}d\mu
\end{equation*}
and thus 
\begin{equation*}
    \int_{V}\left(|\nabla u_{1}|^{2}-\alpha u_{1}^{2}\right)d \mu< 0.
\end{equation*}
For any $t>0$, one has
\begin{equation}
    J_{ \alpha ,\beta }(tu_{1})=\frac{t^{2}}{2}\int_{V}\left(|\nabla u_{1}|^{2}-\alpha u_{1}^{2}\right)d \mu-\beta \log\int_{V}he^{tu_{1}}d \mu.
\label{tu_1}
\end{equation}

${}$

{\bf Case $\boldsymbol{1}$: $\boldsymbol{\beta > 0}$.} Since $\displaystyle{\int_{V}u_{1}d\mu =0}$ and $u_{1}\not\equiv 0$, it follows that there exists $x_{0}\in V$ such that 
\begin{equation*}
   u_{1}(x_{0})>0.
\end{equation*}
Hence
\begin{equation}
\begin{aligned}
    J_{ \alpha ,\beta }(tu_{1})&\leq -\beta \log \Big{(}h(x_{0})\mu (x_{0})\exp \left ( tu_{1}(x_{0})\right ) \Big{)}\\
    &\leq -\beta u_{1}(x_{0})t-\beta \log\Big{(}h(x_{0})\mu (x_{0}) \Big{)}\\
    &\rightarrow -\infty \ (t\rightarrow +\infty ).
\end{aligned}
\label{>0}
\end{equation}
${}$

{\bf Case $\boldsymbol{2}$: $\boldsymbol{\beta\leq 0}$.} Considering \eqref{infty-2}, we have
\begin{equation*}
\begin{aligned} 
    \int_{V}e^{tu_{1}}d\mu &\leq \text{Vol}(V) \exp \left ( t\left \| u_{1} \right \|_{L^{\infty}(V)} \right )\\
    &\leq \text{Vol}(V) \exp \left ( t\left \| u_{1} \right \|_{L^{2}(V)} \mu _{\min }^{-\frac{1}{2}}\right ). \\
\end{aligned}
\end{equation*}
Hence
\begin{equation}
    \int_{V}he^{tu_{1}}d\mu \leq \Big{(}\max _{x\in V}h(x)\Big{)}\text{Vol}(V)\exp \left ( t\left \| u_{1} \right \|_{L^{2}(V)} \mu _{\min }^{-\frac{1}{2}}\right ).
\label{he^u23}
\end{equation}
It follows from \eqref{tu_1} and \eqref{he^u23} that 
\begin{equation*}
\begin{aligned}
    J_{ \alpha ,\beta }(tu_{1})&\leq\frac{t^{2}}{2}\int_{V}\left(|\nabla u_{1}|^{2}-\alpha u_{1}^{2}\right)d \mu -\beta \left \| u_{1} \right \|_{L^{2}(V)}\mu _{\min }^{-\frac{1}{2}} t+C\\
    &\rightarrow -\infty \ \ \ (\text{as} \ t\rightarrow +\infty ),
\end{aligned}
\end{equation*}
where $C$ is some constant depending only on $\beta, h$ and $G$.

${}$

Therefore, we obtain $\displaystyle{\inf_{u\in{\bf H}} J_{ \alpha ,\beta }(u) =- \infty}$.

\subsection{Proof of Theorem 2.3: (ii)} 

Before proving (ii) of Theorem 2.3, we give some properties of the functional space $E_{k}$.

\begin{lemma}
    Let $E_{k}^{\perp}$ be defined as \eqref{E_kp}. Then 
\begin{equation}
    E_{k}^{\perp}=E_{\lambda _{k+1}(V)}\bigoplus \cdots \bigoplus E_{\lambda _{m-1}(V)}.\label{p=}
\end{equation}
\end{lemma}

\begin{proof}
    Note that ${\bf H}=E_{\lambda _{1}(V)} \bigoplus \cdots \bigoplus E_{\lambda _{m-1}(V)}=E_{k}\bigoplus E_{k}^{\perp}$. By the definition of $E_{k}$, we have our desired result.
\end{proof}

Define by 
\begin{equation}
    (E_{k}^{\perp })^{\perp }:=\left\{u \in {\bf F} (V) \Bigg| \int_{V} uv\!\ d\mu=0, \ \forall \ v \in E_{k}^{\perp } \right\}\label{E_kpp}
\end{equation}
the complement space of $E^{\perp}_{k}$.

\begin{lemma}
    Let $(E_{k}^{\perp })^{\perp }$ be defined as \eqref{E_kpp}. Then 
\begin{equation}
    (E_{k}^{\perp })^{\perp }=\left\{{\rm const}\right\}\bigoplus E_{k}.\label{pp=}
\end{equation}
\end{lemma}

\begin{proof}
Obviously, it suffices to prove that $(E_{k}^{\perp })^{\perp } \subseteq\{{\rm const}\}\bigoplus E_{k}$. For any $v \in (E_{k}^{\perp })^{\perp }$, $(v-\overline{v})\in{\bf H}$. Assume 
$v-\overline{v}=w_{1}+w_{2},$
where $w_{1} \in E_{k}$ and  $w_{2} \in E_{k}^{\perp }$. Then we obtain
$$
\int_{V}w_{2}^{2}\!\ d\mu =\int_{V}(v-\overline{v}-w_{1})w_{2}\!\ d\mu =0.
$$
Hence, $ w_{2}=0 $ and $(v-\overline{v}) \in E_{k}$. Therefore, $v=\overline{v}+(v-\overline{v}) \in\{{\rm const}\}\bigoplus E_{k}$.
\end{proof}

${}$

{\it The proof of (ii) of Theorem 2.3.} Note that 
\begin{equation*}
    \int_{V}\left(|\nabla u_{1}|^{2}- \lambda _{1}(V) u_{1}^{2}\right) d\mu =0.
\end{equation*}
Hence, when $\beta>0 $, we can also obtain \eqref{>0}. Apparently, $J_{ \lambda _{1}(V), 0 }$ has a minimizer $u_{1}$. By the proof of Theorem \ref{Th2}, our desired result follows immediately.

${}$

Next we consider the case $\beta < 0$. For any $0\neq u\in E_{1}^{\perp}$, one has
\begin{equation*}
    u=a_{21}u_{21}+\cdots+a_{2n_{2}}u_{2n_{2}}+\cdots+a_{m-1,1}u_{m-1,1}+\cdots+a_{m-1,n_{m-1}}u_{m-1,n_{m-1}}
\end{equation*}
where $a_{21}, \cdots, a_{2n_{2}},\cdots, a_{m-1,1},\cdots, a_{m-1,n_{m-1}}$ are some constants. Then 
\begin{equation*}
\begin{aligned}
     \left \| u\right \|_{L^{2}(V)}^{2}&=a_{21}^{2}+\cdots+a_{2n_{2}}^{2}+\cdots+a_{m-1,1}^{2}+\cdots+a_{m-1,n_{m-1}}^{2},\\
     \left \| \nabla u\right \|_{L^{2}(V)}^{2}&=\lambda _{2}(V)(a_{21}^{2}+\cdots+a_{2n_{2}}^{2})+\cdots+\lambda _{m-1}(V)(a_{m-1,1}^{2}+\cdots+a_{m-1,n_{m-1}}^{2}).
\end{aligned}
\end{equation*}
This leads to the following estimate
\begin{equation}
\begin{aligned}
  \int_{V}\left(|\nabla u|^{2}- \lambda _{1}(V) u^{2}\right) d \mu 
  &\geq \left(\lambda _{2}(V)-\lambda _{1}(V)\right)\left(a_{21}^{2}+\cdots+a_{2n_{2}}^{2}\right)+\cdots\\
  &+ \left(\lambda _{m-1}(V)-\lambda _{1}(V)\right)\left(a_{m-1,1}^{2}+\cdots+a_{m-1,n_{m-1}}^{2}\right)\\
  &\geq\frac{\lambda _{2}(V)-\lambda _{1}(V)}{\lambda _{2}(V)} \left \| \nabla u\right \|_{L^{2}(V)}^{2}.
\label{2-1}
\end{aligned}
\end{equation}

It follows from  \eqref{intheu} and \eqref{2-1} that 
\begin{equation}
    J_{ \lambda _{1}(V),\beta }(u)\geq \frac{\lambda _{2}(V)-\lambda _{1}(V)}{2\lambda _{2}(V)}\left \|\nabla u \right \|_{L^{2}(V)}^{2}+\beta C_{1} \left \| \nabla u \right \|_{L^{2}(V)}+C_{2}
\label{Jlambda1}
\end{equation}
where $C_{1}$ and $C_{2}$ are some constants depending only on $ \beta, h$ and $G$. Hence, we obtain
\begin{equation*}
      J_{ \lambda _{1}(V),\beta }(u)\geq C_{2},
\end{equation*}
where $C_{2}$ is some constant depending only on $ \beta, h$ and $G$. 
This permits us to consider 
\begin{equation*}
    b=\inf_{u\in E_{1}^{\perp}}J_{ \lambda _{1}(V),\beta }(u).
\end{equation*}
Take a sequence of functions $\{u_{j}\}_{j}\subseteq E_{1}^{\perp}$ such that $J_{\lambda _{1}(V) ,\beta }\left ( u_{j}\right )\rightarrow b$. It follows from \eqref{Jlambda1} that there exists a constant $\widetilde{C}$ depending only on $ \beta, h$ and $G$ such that 
$$\left \| \nabla u_{j}\right \|_{L^{2}(V)}\leq \widetilde{C}.$$
By Lemma \ref{Le2}, $\{u_{j}\}_{j}$ is bounded in ${\bf W}^{1,2}(V)$. By Lemma \ref{Le1}, up to a subsequence, $u_{j}\rightarrow u$ in ${\bf W}^{1,2}(V)$. It is easy to see that $u\in E_{1}^{\perp}$ and $J_{ \alpha ,\beta }(u)= b$. 
Based on variational principle, we have
\begin{equation*}
     \Delta u+\alpha u+\frac{\beta h e^{u} }{\displaystyle{\int_{V}he^{u}d \mu}} \in  (E_{1}^{\perp})^{\perp}.
\end{equation*}
By Lemma \ref{pp=}, we obtain
\begin{equation}
      \Delta u+\alpha u=-\frac{\beta he^{u}}{\displaystyle{\int_{V}he^{u}d \mu }}+\xi+\sum_{i=1}^{n_{1}}t_{i} u_{1i}
\label{eta}
\end{equation}
where $\xi $ and $t_{i}$ are constants. Integrating \eqref{eta} on both sides, we have
$$
\xi =\frac{\beta }{{\rm Vol}(V)}. 
$$
Multiplying $u_{1i}$ on both sides of \eqref{eta} and then integrating it on $V$, we obtain 
\begin{equation*}
    t_{i}=\frac{\displaystyle{\beta \int_{V}hu_{1i}e^{u}d\mu}}{\displaystyle{\int_{V}he^{u}d\mu}}.
\end{equation*}
Therefore, the Kazdan-Warner equation \eqref{kw_4} has a solution $u\in  E_{1}^{\perp}$.


\section{Proof of Theorem 2.6}

It follows from \eqref{lambdak+1} that if $\alpha <  \lambda _{k+1}(V)$, then
\begin{equation*}
      \left \|u \right \|_{k+1, \alpha }:=\left[\int_{V}\left(|\nabla u|^{2}- \alpha u^{2}\right)d \mu\right]^{\frac{1}{2}}
\end{equation*}
defines a Sobolev norm on ${\bf H}$. Since ${\bf H}$ is a finite dimensional linear space, $  \left \| \cdot \right \|_{k+1, \alpha }$ is equivalent to the Sobolev norm $\left \|  \cdot \right \|_{{\bf W}^{1,2}(V)}$. 

${}$

{\it The proof of Theorem 2.6.} We first prove that the functional  $J_{ \alpha ,\beta }(u)$ is bounded in $E_{k}^{\perp} $.

${}$

{\bf Case $\boldsymbol{1}$: $\boldsymbol{\beta \geq 0}$.} Assume $0 \neq u\in  E_{k}^{\perp }$. Let
$$
v=\frac{u}{||u||_{k+1, \alpha }}. 
$$
Then $v\in E_{k}^{\perp }$ and $\left \| v \right \|_{k+1, \alpha }=1$.
We claim that there exists a constant $C$ depending only on $k,\alpha ,\theta $ and $G$ such that 
\begin{equation*}
\int_{V}e^{\theta  v^{2}}d \mu  \leq C\left ( k,\alpha ,\theta ,G\right ).  
\end{equation*}
Indeed, since $ \left \| \cdot \right \|_{k+1, \alpha }$ is equivalent to the Sobolev norm $|| \cdot ||_{{\bf W}^{1,2}(V)}$, there exists a constant $K$ depending only on $k,\alpha $ and $G$ such that for all functions $v\in {\bf H}$,
\begin{equation}
    \int_{V}|\nabla v|^{2}d \mu  \leq K\left \| v \right \|_{k+1, \alpha }^{2}.
\label{k+1,a}
\end{equation}
Let 
$$
\widetilde{v}=\frac{v}{\sqrt{K}|| v ||_{k+1, \alpha }}. 
$$
Then $\widetilde{v} \in  E_{k}^{\perp }$ and $\displaystyle{\int_{V}|\nabla \widetilde{v}|^{2}d \mu  \leq 1}$. By Lemma \ref{Le3}, for any $\theta  \in  \mathbb{R} $, one can find a constant $C$ depending only on $\theta $ and $G$ such that 
\begin{equation*}
\int_{V}e^{\theta  \widetilde{v}^{2}}d \mu  \leq C\left ( \theta  ,G\right ).    
\end{equation*}
Hence
\begin{equation*}
   \int_{V}e^{\theta  v^{2}}d \mu =\int_{V}\exp \left (\theta  \widetilde{v}^{2}\cdot K\left \| v \right \|_{k+1, \alpha }^{2} \right )d \mu  \leq e^{K}C\left ( \theta  ,G\right ).
\end{equation*}
For any $\varepsilon > 0$, one has
\begin{equation*}
\begin{aligned}
     \int_{V}e^{u}d\mu &\leq \int_{V}\exp \left (\varepsilon \left \| u \right \|_{k+1, \alpha }^{2}+\frac{u^{2}}{4 \varepsilon \left \| u \right \|_{k+1, \alpha }^{2}} \right )d\mu\\
     &=\exp \left (\varepsilon \left \| u \right \|_{k+1, \alpha }^{2} \right )\int_{V}\exp \left (\frac{u^{2}}{4\varepsilon \left \| u \right \|_{k+1, \alpha }^{2}} \right )d\mu  \\
     &\leq C\exp \left ( \varepsilon \left \| u \right \|_{k+1, \alpha }^{2}\right )
\end{aligned}
\end{equation*}
where $C$ is some constant depending only on $ k,\alpha ,\varepsilon $ and $G$.
Hence, 
\begin{equation}
    \int_{V}he^{u}d\mu \leq C\left ( \max _{x\in V}h(x)\right )\exp \left ( \varepsilon \left \| u \right \|_{k+1, \alpha }^{2}\right ).
\label{heuk}
\end{equation}
It follows from \eqref{J_a,b} and \eqref{heuk} that 
\begin{equation*}
    J_{ \alpha ,\beta }(u)\geq \frac{1}{2} \left \| u \right \|_{k+1, \alpha }^{2}-\beta \varepsilon  \left \| u \right \|_{k+1, \alpha }^{2}+C_{1}
\end{equation*}
where $C_{1}$ is some constant depending only on $ \varepsilon ,\alpha,\beta, h$ and $G$.
Choosing $\displaystyle{\varepsilon =\frac{1}{4\beta}}$ and considering \eqref{k+1,a}, we have
\begin{equation}
      J_{ \alpha ,\beta }(u)\geq \frac{1}{4} \int_{V}\left \| \nabla u \right \|_{L^{2}(V)}+C_{1}.
\label{J_a,b31}
\end{equation}

${}$

{\bf Case $\boldsymbol{1}$: $\boldsymbol{\beta < 0}$.} It follows from \eqref{J_a,b}, \eqref{k+1,a} and  \eqref{intheu} that 
\begin{equation}
     J_{ \alpha ,\beta }(u) \geq \frac{1}{2K}\left \| \nabla u \right \|_{L^{2}(V)}^{2}+C_{2}\left \| \nabla u \right \|_{L^{2}(V)}+C_{3}
\label{J_a,b32}
\end{equation}
where $C_{2}, C_{3}$ are some constants depending only on $ \beta, h$ and $G$. Hence, we obtain
\begin{equation*}
      J_{ \alpha ,\beta }(u)\geq C_{4}.
\end{equation*}

${}$

Therefore  $J_{ \alpha ,\beta }(u)$ has a lower bound on $H$. This allows us to consider 
\begin{equation*}
    b:=\inf_{u\in E_{k}^{\perp }}J_{ \alpha ,\beta }(u).
\end{equation*}
Take a sequence of functions $\{u_{j}\}_{j}\subseteq E_{k}^{\perp }$ such that $J_{\alpha ,\beta }\left ( u_{j}\right )\to b$. It follows from \eqref{J_a,b31} and \eqref{J_a,b32} that for any $\beta \in \mathbb{R}$, there exists a constant $\widetilde{C}$ depending only on $ \alpha,\beta, h$ and $G$ such that 
\begin{equation*}
    \left \| \nabla u_{j}\right \|_{L^{2}(V)}\leq \widetilde{C}.
\end{equation*}
By Lemma \ref{Le2}, $\{u_{j}\}_{j}$ is bounded in ${\bf W}^{1,2}(V)$. By Lemma \ref{Le1}, up to a subsequence, $u_{j}\to u$ in ${\bf W}^{1,2}(V)$. It is easy to see that  $u\in  E_{k}^{\perp }$ and $J_{ \alpha ,\beta }(u)= b$. Based on variational principle, we have
\begin{equation*}
\Delta u+\alpha u+\frac{\beta he^{u}}{\displaystyle{\int_{V}he^{u}d \mu}} \in (E_{k}^{\perp })^{\perp }.
\end{equation*}
It follows from  Lemma \ref{E_kpp} that
\begin{equation}
     \Delta u+\alpha u=-\frac{\beta he^{u}}{\displaystyle{\int_{V}he^{u}d \mu }}+\xi+\sum_{s=1}^{k}\sum_{i=1}^{n_{s}}t_{si}u_{si}
\label{xi3}
\end{equation}
where $\xi$ and $t_{si}$ are some constants. Integrating \eqref{xi3} on both sides, we have
$$
\xi =\frac{\beta }{{\rm Vol}(V)}. 
$$
Multiplying $u_{si}$ on both sides of \eqref{xi3} and then integrating it on $V$, we obtain 
$$t_{si}=\frac{\displaystyle{\beta \int_{V}hu_{si}e^{u}d\mu}}{\displaystyle{\int_{V}he^{u}d\mu}}.
$$
Therefore, the Kazdan-Warner equation \eqref{kw_3} has a solution $u\in E_{k}^{\perp }$.

\section{Proof of Theorem 2.7}

In this section, we consider the case that $\alpha\geq\lambda_{k+1}(V)$.

\subsection{Proof of Theorem 2.7: (i)}

Fix $\alpha > \lambda _{k+1}(V)$. Take $u_{k+1}$ as an eigenfunction with respect to $\lambda _{k+1}(V)$ then $u_{k+1}\in E_{k}^{\perp } $ . Obviously, we have 
\begin{equation*}
    \int_{V}|\nabla u_{k+1}|^{2}d \mu=\lambda _{k+1}(V)\int_{V}u_{k+1}^{2}d\mu
\end{equation*}
and thus 
\begin{equation*}
    \int_{V}\left(|\nabla u_{k+1}|^{2}-\alpha u_{k+1}^{2}\right)d \mu<0.
\end{equation*}
For any $t>0$, one arrives at 
\begin{equation}
     J_{ \alpha ,\beta }(tu_{k+1})=\frac{t^{2}}{2}\int_{V}\left(|\nabla u_{k+1}|^{2}-\alpha u_{k+1}^{2}\right)d \mu-\beta \log\int_{V}he^{tu_{k+1}}d \mu.
\label{tuk+1}
\end{equation}

${}$

{\bf Case $\boldsymbol{1}$: $\boldsymbol{\beta > 0}$.} Since $\displaystyle{\int_{V}u_{k+1}\!\ d\mu =0}$ and $u_{k+1}\not\equiv0$, it follows that there exists $x_{0}\in V$ such that 
\begin{equation*}
    u_{k+1}(x_{0})>0.
\end{equation*}
Hence
\begin{equation}
\begin{aligned}
     J_{ \alpha ,\beta }(tu_{k+1})
     &\leq -\beta \log \Big{(}h(x_{0})\mu(x_{0})\exp({tu_{k+1}(x_{0})}\Big{)}\\
     &\leq -\beta u_{k+1}(x_{0})t-\beta \log \Big{(}h(x_{0})\mu(x_{0})\Big{)}\\
     &\rightarrow -\infty \ (t\rightarrow +\infty ).
\end{aligned}
\label{k+1}
\end{equation}

${}$

{\bf Case $\boldsymbol{1}$: $\boldsymbol{\beta \leq 0}$.} Considering \eqref{infty-2}, we have
\begin{equation*}
\begin{aligned}
    \int_{V}e^{tu_{k+1}}d\mu &\leq \text{Vol}(V) \exp\left (t\left \| u_{k+1} \right \|_{L^{\infty}(V)}\right )\\
    &\leq \text{Vol}(V) \exp\left (t\left \| u_{k+1} \right \|_{L^{2}(V)}\mu _{\min }^{-\frac{1}{2}}\right ).
\end{aligned}
\end{equation*}
Hence
\begin{equation}
    \int_{V}he^{tu_{k+1}}d\mu \leq \text{Vol}(V)\left ( \max _{x\in V}h(x)\right )\exp\left (t\left \| u_{k+1} \right \|_{L^{2}(V)}\mu _{\min }^{-\frac{1}{2}}\right ).
\label{he^u33}
\end{equation}
It follows from \eqref{tuk+1} and \eqref{he^u33} that 
\begin{equation*}
\begin{aligned}
  J_{ \alpha ,\beta }(tu_{k+1})&\leq\frac{t^{2}}{2}\int_{V}(|\nabla u_{k+1}|^{2}-\alpha u_{k+1}^{2})d \mu-\beta \left \|  u_{k+1} \right \|_{L^{2}(V)}\mu _{\min }^{-\frac{1}{2}}t+C_{2}\\
  &\rightarrow -\infty \ \ \ ( \text{as} \ t\rightarrow +\infty )
\end{aligned} 
\end{equation*}
where $C_{2}$ is some constant depending only on $\beta, h$ and $G$.

${}$

Therefore, we obtain $\displaystyle{\inf_{u\in E_{k}^{\perp }} J_{ \alpha ,\beta }(u) =- \infty}$.

\subsection{Proof of Theorem 2.7: (ii)}

Note that 
\begin{equation*}
    \int_{V}\left(|\nabla u_{k+1}|^{2}- \lambda _{k+1}(V) u_{k+1}^{2}\right)d \mu =0.
\end{equation*}
Hence, when $\beta>0 $, we can also obtain \eqref{k+1}. Apparently, $J_{ \lambda _{k+1}(V), 0 }$ has a minimizer $u_{k+1}$. By the proof of Theorem \ref{Th4}, our desired result follows immediately.

${}$

Next we consider the case $\beta < 0$. For any $0\neq u\in E_{k+1}^{\perp}$, 
\begin{eqnarray*}   u&=&\left(a_{k+2,1}u_{k+2,1}+\cdots+a_{k+2,k+2}u_{k+2,n_{k+2}}\right)+\cdots\\
     &&+ \ \left(a_{m-1,1}u_{m-1,1}+\cdots+a_{m-1,n_{m-1}}u_{m-1,n_{m-1}}\right)
\end{eqnarray*}
where 
\begin{equation*}
    a_{k+2,1},\cdots, a_{k+2,n_{k+2}} ,\cdots, a_{m-1,1},\cdots, a_{m-1,n_{m-1}}
\end{equation*}
are some constants. Then 
\begin{equation*}
\begin{aligned}
     \left \| u\right \|_{L^{2}(V)}^{2}&=\left(a_{k+2,1}^{2}+\cdots+a_{k+2,n_{k+2}}^{2}\right)+\cdots+\left(a_{m-1,1}^{2}+\cdots+a_{m-1,n_{m-1}}^{2}\right),\\
     \left \| \nabla u\right \|_{L^{2}(V)}^{2}&=\lambda _{k+2}(V)\left(a_{k+2,1}^{2}+\cdots+a_{k+2,n_{k+2}}^{2}\right)+\cdots\\
     &+\lambda _{m-1}(V)\left(a_{m-1,1}^{2}+\cdots+a_{m-1,n_{m-1}}^{2}\right).
\end{aligned}
\end{equation*}
This leads to the following estimate
\begin{equation}
\begin{aligned}
  &\int_{V}\left(|\nabla u|^{2}- \lambda _{k+1}(V) u^{2}\right) d \mu\\ 
  &\geq \left(\lambda _{k+2}(V)-\lambda _{k+1}(V)\right)\left(a_{k+2,1}^{2}+\cdots+a_{k+2,n_{k+2}}^{2}\right)+\cdots\\
  &+ \left(\lambda _{m-1}(V)-\lambda _{k+1}(V)\right)\left(a_{m-1,1}^{2}+\cdots+a_{m-1,n_{m-1}}^{2}\right)\\
  &\geq\frac{\lambda _{k+2}(V)-\lambda _{k+1}(V)}{\lambda _{k+2}(V)} \left \| \nabla u\right \|_{L^{2}(V)}^{2}.
\label{k2-1}
\end{aligned}
\end{equation}
It follows from \eqref{k2-1} and \eqref{intheu} that 
\begin{equation}
    J_{ \lambda _{k+1}(V),\beta }(u)\geq \frac{\lambda _{k+2}(V)-\lambda _{k+1}(V)}{2\lambda _{k+2}(V)}\left \|\nabla u \right \|_{L^{2}(V)}^{2}+\beta C_{1} \left \| \nabla u \right \|_{L^{2}(V)}+C_{2}
\label{Jlambdak+1}
\end{equation}
where $C_{1}$ and $C_{2}$ are some constants depending only on $ \beta, h$ and $G$. Hence, we obtain
\begin{equation*}
      J_{ \lambda _{1}(V),\beta }(u)\geq C_{3},
\end{equation*}
where $C_{3}$ is some constant depending only on $ \beta, h$ and $G$. This allows us to consider 
\begin{equation*}
    b:=\inf_{u\in E_{k+1}^{\perp}}J_{ \lambda _{k+1}(V),\beta }(u).
\end{equation*}
Take a sequence of functions $\{u_{j}\}_{j}\subseteq  E_{k+1}^{\perp}$ such that $J_{\lambda _{k+1}(V) ,\beta }\left ( u_{j}\right )\to b$. It follows from \eqref{Jlambda1} that there exists a constant $\widetilde{C}$ depending only on $ \beta,h$ and $G$ such that 
$$\left \| \nabla u_{j}\right \|_{L^{2}(V)}\leq \widetilde{C}.$$
By Lemma \ref{Le2}, $\{u_{j}\}_{j}$ is bounded in ${\bf W}^{1,2}(V)$. By Lemma \ref{Le1}, up to a subsequence, $u_{j}\to u$ in ${\bf W}^{1,2}(V)$. It is easy to see that  $u\in   E_{k+1}^{\perp}$ and $J_{ \alpha ,\beta }(u)= b$.
Based on variational principle, we have
\begin{equation*}
\Delta u+\alpha u+\frac{\beta he^{u}}{\displaystyle{\int_{V}he^{u}d \mu}} \in (E_{k+1}^{\perp })^{\perp }.
\end{equation*}
It follows from  Lemma \ref{E_kpp} that
\begin{equation}
     \Delta u+\alpha u=-\frac{\beta he^{u}}{\displaystyle{\int_{V}he^{u}d \mu }}+\xi+\sum_{s=1}^{k+1}\sum_{i=1}^{n_{s}}t_{si}u_{si}
\label{xi6}
\end{equation}
where $\xi$ and $t_{si}$ are some constants. Integrating \eqref{xi6} on both sides, we have
$$
\xi =\frac{\beta }{\text{Vol}(V)}. 
$$
Multiplying $u_{si}$ on both sides of \eqref{xi6} and then integrating it on $V$, we obtain 
$$t_{si}=\frac{\displaystyle{\beta \int_{V}hu_{si}e^{u}d\mu}}{\displaystyle{\int_{V}he^{u}d\mu}}.$$
Therefore, the Kazdan-Warner equation \eqref{kw_5} has a solution $u\in E_{k+1}^{\perp} $.



\begin{thebibliography}{16}

\bibitem{heat}
Chen, Wenxiong; Li, Congming.
\textit{Gaussian curvature in the negative case.}
Proceedings of the American Mathematical Society {\bf 131}(2003), no. 3, 741-744. MR1937411

\bibitem{equation}
Ding, Weiyue; Jost, Jurgen; Li, Jiayu; Wang, Guofang. \textit{The differential equation $\Delta u=8\pi -8\pi he^{u}$ on a compact Riemann surface.}
Asian Journal of Mathematics {\bf 1}(1997), no. 2, 230-258. MR1491984

\bibitem{borderline}
Fontana, Luigi.
\textit{Sharp borderline Sobolev inequalities on compact Riemannian manifolds.}
Commentarii Mathematici Helvetici {\bf 68}(1993), no. 1, 415-454. MR1236762

\bibitem{negative}
Ge, Huabin.
\textit{Kazdan-Warner equation on graph in the negative case.}
Journal of Mathematical Analysis and Applications {\bf 453}(2017), no. 2, 1022-1027. MR3648273


\bibitem{infinite}
Ge, Huabin; Jiang, Wenfeng.
\textit{Kazdan-Warner equation on infinite graphs.}
Journal of the Korean Mathematical Society {\bf 55}(2018), no. 5, 1950052, 17 pp. MR3849352


\bibitem{p}
Ge, Huabin.
\textit{The $p$-th Kazdan-Warner equation on graphs.}
Communications in Contemporary Mathematics {\bf 22}(2020), no. 6, 1091-1101. MR3818123

\bibitem{graph}
Grigor'yan, Alexander; Lin, Yong; Yang, Yunyan. \textit{Kazdan--Warner equation on graph.}
Calculus of Variations and Partial Differential Equations {\bf 55}(2016), no. 4, 92, 13 pp. MR4416135


\bibitem{curvature}
Kazdan, Jerry L.; Warner, F. W.. \textit{Curvature functions for compact $2$-manifolds.}
Annals of Mathematics {\bf 99}(1974), no. 1, 14-47. MR0343205


\bibitem{extremal}
Yang, Yunyan. \textit{Extremal functions for Trudinger–Moser inequalities of Adimurthi–Druet type in dimension two.}
Journal of Differential Equations {\bf 258}(2015), no. 9, 3161-3193. MR3317632

\bibitem{blow}
Yang, Yunyan; Zhu, Xiaobao.
\textit{Blow-up analysis concerning singular Trudinger–Moser inequalities in dimension two.}
Journal of Functional Analysis {\bf 272}(2017), no. 8, 3347-3374. MR3614172

\bibitem{surface}
Yang, Yunyan; Zhu, Xiaobao.
\textit{Existence of solutions to a class of Kazdan-Warner equations on compact Riemannian surface.}
Science China Mathematics {\bf 61}(2018), no. 6, 1109-1128. MR3802763

\end{thebibliography}
\end{document}